\documentclass[11pt]{amsart}
\usepackage{geometry}          
\usepackage{graphicx}
\usepackage{amssymb}
\usepackage[utf8]{inputenc}     
\usepackage{comment}
\usepackage{a4wide}
\usepackage{amsfonts}
\usepackage{amsthm}
\usepackage{amssymb}
\usepackage{amsmath}
\usepackage[initials]{amsrefs}
\usepackage{enumerate}
\usepackage{mathtools}
\usepackage{url}
\usepackage{tikz}

\usepackage{color}

\newtheorem{theorem}{Theorem}
\newtheorem*{definition}{Definition}
\newtheorem{lemma}{Lemma}

\def\cc#1{\mathcal{#1}}
\newcommand{\N}{{\mathbb N}}

\newcommand{\D}[1]{D^{\tiny \text{\it #1}}}
\newcommand{\cf}{\!\text{\it cf}}
\newcommand{\start}{\text{\it start}}
\newcommand{\Tau}{\mathcal T}

\everymath{\displaystyle}

\author[V. Becher]{Verónica Becher}
\address[V. Becher]{
\noindent  Departmento de  Computaci\'on,   Facultad de Ciencias Exactas y Naturales\newline
 Universidad de Buenos Aires \& ICC CONICET\newline
Pabell\'on I, Ciudad Universitaria, 1428 Buenos Aires, Argentina}
\email{vbecher@dc.uba.ar}

\author[M. G. Madritsch]{Manfred G. Madritsch}
\address[M. G. Madritsch]{
\noindent 1. Universit\'e de Lorraine, Institut Elie Cartan de Lorraine, UMR 7502, Vandoeuvre-l\'es-Nancy, F-54506, France;\newline
\noindent 2. CNRS, Institut Elie Cartan de Lorraine, UMR 7502, Vandoeuvre-l\'es-Nancy, F-54506, France}
\email{manfred.madritsch@univ-lorraine.fr}
\title{On a question of Mendès France on normal numbers}

\date{\today}
\begin{document}

\maketitle
\bigskip

\begin{abstract}
  In 2008 or earlier, Michel Mendès France asked for an instance of a 
  real number~$x$ such that both~$x$ and~$1/x$ are simply
  normal to a given integer base~$b$. We give a positive answer to
  this question by  constructing a number $x$ such that both
  $x$ and its reciprocal $1/x$ are continued fraction normal as well
  as normal to all integer bases greater than or equal to~$2$.
  Moreover, $x$ and $1/x$ are both computable.
\end{abstract}
\bigskip

\bigskip

%{\footnotesize \tableofcontents}

\section{Introduction and statement of results}

In this note we solve a problem posed by Michel Mendès France asking
for an  instance  of a real number~$x$ such that both~$x$
and~$1/x$ are simply normal to a given integer base~$b$.  The problem
appeared in the literature in 2008 in~\cite{Rivoal} and it was
presented to us by Gerhard Larcher.

The continued fraction representation of a positive number and its
reciprocal are identical except for a shift one place left or right
depending on whether the number is less than~$1$ or greater than~$1$,
respectively.  That is, the numbers represented by
${\displaystyle [a_{0};a_{1},a_{2},\ldots} ]$ and
${\displaystyle [0;a_{0},a_{1},\ldots ]} $ are reciprocals. 
This fact allows us to prove the following extension of the problem of Mendès France.

\begin{theorem}\label{thm}
  We give a construction of  a number $x$ such that both $x$ and its reciprocal
  $1/x$ are continued fraction normal and absolutely normal.
  Moreover, they are both computable.
\end{theorem}

To construct $x$ and $1/x$ we define incrementally their continued
fraction expansions.  To ensure that both $x$ and $1/x$ are
continued fraction normal and absolutely normal we follow the work by
Becher and Yuhjtman in~\cite{BY}, where they construct a number $x$
which is continued fraction normal and absolutely normal. The
challenge in the present paper is to handle simultaneously \emph{two}
constructions, one for $x$ and one for $1/x$. These constructions work
by defining successive refinements of appropriate subintervals to
achieve, in the limit, in both cases, continued fraction normality and
simple normality to all integer bases. At each step the choice of
digits for the two constructions is done without revisiting the digits
chosen at previous steps. The difficulty is to choose the same digits
for the continued fraction expansions for $x$ and for $1/x$.

\section{Two types of normality}

In this section we follow the standard notation in this area. For a
detailed account on normal numbers  see~\cite{Kuipers2006,
  drmotatichy1997, Bugeaud2012, BCchapter},
  for symbolic dynamics  see~\cite{lindmarcus1995, DK}
and for a combination of both see~\cite{madritsch2018}.

As usual we write $\N=\{1,2,3,\ldots\}$ to denote the set of positive
integers and $\N^k$ to denote the set of $k$ tuples of positive
integers.  For a finite set $S$, we denote by $\#S$ its
cardinality. Similarly for an infinite set $S$ of real numbers, $|S|$
denotes its Lebesgue measure; hence, if $S$ is an interval in the real
line, then $|S|$ is its length. We use Landau's notation for the
asymptotic behaviour of functions. Thus a function $g(x)=O(f(x))$ if
there exist constants $x_0$ and $c$ such that for every $x\geq x_0$,
$|g(x)| < c \cdot |f(x)|$.  We write $\log$ to denote the logarithm to
base~$e$.

\subsection{On continued fraction normality}

For a real number $x$ in the unit interval, the continued fraction
expansion of $x$ is the integer part $a_0=\lfloor x\rfloor$ together with a sequence of
positive integers $a_1, a_2, \ldots $, such that
\[
x=a_0+ 
\cfrac{1}{a_1 + \cfrac{1}{a_2 + \cfrac{1}{\ddots \, + \cfrac{1}{a_n + \cfrac{1}{\ddots} }   }}}
\]
and we write $x=[a_0;a_1,a_2,\ldots]$ for short.  This expansion of a
number can be seen as an infinite word over the alphabet $\N$.
Since normality is an asymptotic property of the digits we drop the
integer part of the continued fraction representation in the sequel
and write $[a_1, a_2 \ldots ]$ instead of $[a_0; a_1, a_2, \ldots]$.

A way of obtaining the continued fraction expansion is applying the
Gauss map $T\colon [0,1]\to[0,1]$ defined by
\[T(x)=\begin{cases}\frac1x-\left\lfloor\frac1x\right\rfloor&\text{if }x\neq0,\\
    0&\text{otherwise.}\end{cases}
\]
If $x=[a_1,a_2,\ldots]$ then $T^n(x)=[a_{n+1}, a_{n+2}, \ldots]$ and
for every~$n\geq 1$, $a_n=\lfloor 1/T^{n-1}(x)\rfloor$.  Otherwise
said, the Gauss map corresponds to the left shift in the associated
symbolic dynamical system over the alphabet $\mathbb{N}$.

The map $T$ possesses an invariant ergodic measure,
the Gauss measure $\mu$, which is absolutely continuous with respect
to Lebesgue measure (\textit{cf.} Dajani and Kraaikamp \cite{DK}). 
In particular, for every Lebesgue measurable set $A$, we have
\[
\mu(A)= \frac{1}{\log 2}\int_A \frac{1}{1+x} \ {\rm d}x.
\]
An interval $I$ in the unit interval is a  cylinder set of order $n$ with respect
to the continued fraction expansion, or  \cf-ary  of order $n$,  
if there is a finite
continued fraction $[a_1,\dotsc, a_n]$ such that the interval $I$ is
equal to the set of all the numbers whose first $n$ digits of their
continued fraction expansion are $a_1,\dotsc, a_n$.  Thus,
\begin{align*}
I_{[a_1, \dotsc, a_n]} &= ([a_1, \dotsc, a_n], [a_1, \dotsc, a_n +1]), \mbox{ or }
\\
I_{[a_1, \dotsc, a_n]} &= ([a_1, \dotsc, a_n+1], [a_1, \dotsc, a_n]) 
\end{align*}
depending on whether $n$ is even or odd, respectively. The
set of \cf-ary intervals of order~$n$ form a partition of the unit
interval in infinitely many parts of different lengths.

A real number $x=[a_1,a_2,\ldots]$ is continued fraction normal
(or \cf-normal for short) if every word of positive integers occurs in
its continued fraction expansion with the asymptotic frequency
determined by the Gauss measure. Otherwise said, $x$ is generic for
$\mu$, \textit{i.e.} for every positive integer
$k$ and for every word $v_1\dotsc v_k$ in $\mathbb N^k$, we have
\[
\lim_{n\to\infty}\frac{1}{n} \#\{ j: 1\leq j\leq n, a_{j}=v_1, \dotsc, a_{j+k-1}=v_k\}
=\mu(I_{[v_1, \dotsc, v_k]}).
\]
In order to get a feeling for \cf-normality we provide some remarks.
%First note that 
All quadratic irrationals are not \cf-normal,
because their expansions are periodic. However, nothing particular is
known for algebraic numbers of higher degree. The number
$e= [2;1,2,1,1,4,1,1,6,1,1,8,\ldots]$ is not \cf-normal because it is
the concatenation of the pattern $(1 m 1)$, for all even~$m$ in
increasing order, and no other odd digit except~$1$ occurs in the
expansion.  Nothing else is known about \cf-normality of other
transcendental constants. By Birkhoff's Ergodic
Theorem~\cite{Birkhoff1931}, almost every real in the
unit interval is \cf-normal and there are several constructions of \cf-normal numbers.

\subsection{On normality to integer bases}

For an integer $b\geq2$ called the base we denote by
$\mathcal{N}_b=\{0,\ldots,b-1\}$ the corresponding set of digits. % We
% call $b$ the base and $\mathcal{N}_b$ the set of digits in base $b$.
Then, every positive integer $n$ has a unique representation of the
form
\[n=a_\ell b^{\ell}+\cdots+a_1b+a_0\] with $a_i\in\mathcal{N}_b$ for
$0\leq i\leq\ell$.  This representation can be extended to real
numbers $x$ in $[0,1]$~by
\begin{gather*}\label{bary-expansion} 
x = \sum_{i=1}^{\infty} a_i b^{-i}
\end{gather*}
with $a_i\in\mathcal{N}_b$ for $i\geq1$ and $a_i\neq b-1$ infinitely
often. The latter ensures that every rational number has a unique
representation (the greedy one).

As in the case of continued fraction expansions there exists a map in
the unit interval that describes the dynamic aspect of the $b$-ary
expansion. For a positive integer $b\geq2$ we consider the map
$S_b:[0,1]\to [0,1]$ defined by
\[
S_b(x)=bx-\lfloor bx\rfloor.
\] 

An interval $I$ in the unit interval  is  a cylinder set of order $n$ with respect to the $b$-ary
expansion (or $b$-ary of order $n$ for short) if there is a finite word
$d_1\cdots d_n$ over $\mathcal{N}_b$ such that the interval $I$ is equal
to the set of real numbers whose first $n$ digits of their $b$-ary
expansion are equal to $d_1,\dotsc, d_n$.  The set of $b$-ary
intervals of order $n$ form a partition of the unit interval in
finitely many parts of equal length (in contrast to  the infinitely many
parts of different lengths in the case of the continued fraction expansion).

A real $x=a_1b^{-1}+a_2b^{-2}+\cdots$ is simply normal with
respect to base $b$ if every digit occurs in the $b$-ary expansion of
$x$ with the same asymptotic frequency $1/b$.  That is, for each
$v\in\mathcal{N}_b$,
\[\lim_{n\to\infty}\frac{\#\{1\leq j\leq n\colon
    a_j=v\}}{n}=\frac1b.
\] 
Normality to base $b$ is simple normality to bases
$b, b^2, b^3, \ldots$, all the powers of $b$ (this definition of
normality is equivalent to Borel's original definition
\cite{Borel1909}, the proof is due to Pillai in 1940 \cite[Theorem
4.2]{Bugeaud2012}.)  Absolute normality is normality to every integer
base $b\geq2$; hence, simple normality to every integer base $b\geq2$.
Borel showed that almost all real numbers (with respect to Lebesgue
measure) are absolutely normal. In the same way as above this also
follows from Birkhoff's Ergodic Theorem~\cite{Birkhoff1931}, since the
Lebesgue measure is ergodic with respect to the map $S_b$
(\textit{cf.} Dajani and Kraaikamp~\cite{DK}).

\section{Definitions and Lemmas}

To prove Theorem~\ref{thm}  we give two simultaneous constructions, 
one for $x$ and one  for~$1/x$. 
For each, we follow the construction of a  continued
fraction normal and absolutely normal number of Becher and Yuhjtman
in~\cite{BY}, which in turn is based on the work on aboslutely normal numbers~\cite{poly}. 
For a similar construction for a normal number with respect to all Pisot
bases see Madritsch, Scheerer and Tichy~\cite{MST}.

%First we use the fact that the continued fraction expansion of $x$ and
%$1/x$ differ just by the first  digit. 
%
%\note{I do not understand the next  sentence...   where it says  freedom in the choice ... up to a certan {\em point}}
%To have some freedom in the choice of
%the \cf-digits we fix a \cf-ary interval in each step (which itself
%fixes the \cf-digits of $x$ and $1/x$ up to a certain point). 
%For simple normality with respect to base $b$ we cover the \cf-ary
%interval with $b$-ary intervals whose corresponding expansions are
%close to the expected mean. 
%\note{I think the sentence ''whose corresponding expansions are close to the expected mean is  not good
%because in an interval you have all the possible expansions.}
%Here we let $b$ only run up to a certain
%base $t$, which we increase every now and then in order to have simple
%normality to all bases asymptotically.
%
%
%
%We summarize this concept in the following definition of a $t$-brick.
%\note{I do not think that we summarize the concept in the definition of t-brick.\\
%Instead I suggest:
%\\
%We base our construction in the following   definitions.
%}

\subsection{Definitions}

Each of the two constructions work by defining a sequence of nested
intervals.  For this we introduce the definition of a $t$-brick and a
refinement of a $t$-brick.  To control continued fraction normality we
use \cf-ary intervals and to control normality in each integer base
$b$ we use $b$-ary intervals.

\begin{definition}[$t$-brick]
 For an integer $t \geq 2$, a $t$-brick is a tuple 
$(\sigma_{\cf},\sigma_2, \ldots, \sigma_t)$  as follows
\begin{itemize} 
\item[-]  the interval $\sigma_{\cf}$ is \cf-ary, 
\item[-] for each $b=2, \ldots t$, $\sigma_b$ is either a $b$-ary
  interval or the union of two consecutive $b$-ary intervals of the
  same order;
\item[-] for each $b=2, \ldots t$,
\begin{align*}
\sigma_{\cf} &\subset \sigma_b
\\
{|\sigma_{\cf}|} &\geq \frac{|\sigma_b|}{4\cdot 16e^{4C}b}.
\end{align*}
\end{itemize}
\end{definition}

We use the classical notion of discrepancy, but  not on arbitrary  intervals.
For discrepancy with  respect to continued fraction expansions we consider the classical discrepancy  restricted to
\cf-ary intervals . 
For discrepancy with respect to $b$-ary expansion we consider discrepancy 
restricted to $b$-ary intervals.

\begin{definition}[Discrepancy for continued fraction]
For a finite word
$\mathbf{v}=v_1v_2\ldots v_k$ over the alphabet $\N$ we denote
the discrepancy of $x=[a_1, a_2, \ldots ]$ with respect to $\mathbf{v}$ in the first $n$ positions of its continued
fraction expansion by
\[
\D{\cf-ary}_{\mathbf{v},n}(x)=
  \left|\frac{1}{n} \#\{ j: 1\leq j\leq n, a_{j}=v_1, \dotsc, a_{j+k-1}=v_k\}-
\mu\left(I_{[v_1,\dotsc,v_k]}\right)\right|
=\mu(I_{[v_1, \dotsc, v_k]}).
\]
\end{definition}

Clearly, a real number $x$ is continued fraction normal if and only if
for every positive integer~$k$, and for every word $\mathbf{v}\in\N^k$
of~$k$ positive integers,
\[
\lim_{n\to \infty} \D{\cf-ary}_{\mathbf{v},n}(x)=0.
\]
With some notation abuse we write
$\D{\cf-ary}_{\mathbf{v},n}({\mathbf w})$ for the discrepancy of a
\cf-word ${\mathbf w}$ of positive integers.

In a similar way we define the $b$-ary variant of discrepancy as the
distance of a finite word from uniform distribution of the digits.
\begin{definition}[Discrepancy for integer base representation]
For a real $x=\sum_{j\geq 1} a_j b^{-j}$ we define the discrepancy of the digit $v\in {\mathcal N}_b$ among the first $n$
digits of its $b$-ary expansion %\eqref{bary-expansion} 
by
\[
\D{b-ary}_{v,n}(x)= \left|  \frac{1}{n} \#\{1\leq j\leq n\colon   a_j=v \} -\frac{1}{b}\right|.
\]
and
\[
\D{b-ary}_{n}(x)=\max_{v\in\mathcal{N}_b}\D{b-ary}_{v,n}(x).
\]
\end{definition}

Clearly, a real number $x$ is simply normal to base $b$ if and only if
its expansion in base $b$ is such that
\[
  \lim_{n\to\infty} \D{b-ary}_n(x) = 0.
\] 
With some notation abuse we  write  $\D{b-ary}_{n}({\mathbf w})$  for
the discrepancy of a word ${\mathbf w}$ over the alphabet~$\mathcal{N}_b$.

%Now we provide the essential step in the construction, the refinement.
%
%\note{The refinement is not a step!}

\begin{definition}[Refinement of a $t$-brick]\label{def:refinement}
  A $t$-brick $\vec \sigma =(\sigma_{\cf},\sigma_2,\ldots ,\sigma_t)$
  is refined by a $t'$-brick
  $\vec \tau = (\tau_{\cf}, \tau_2\ldots \tau_{t'})$ if

  \begin{itemize}
  \item[-] $t' = t$ or $t'=t+1$,

  \item[-] $\tau_{\cf} \subseteq \sigma_{\cf}$,

  \item[-] for $b=2, \ldots , t$, $\tau_{b} \subset \sigma_b$.
  \end{itemize}

  The refinement is said to have discrepancy less than $\epsilon$ if
  \begin{enumerate}
  \item[-] the new \cf-word $\mathbf{w}$ corresponding to the
    inclusion $\tau_{\cf} \subset \sigma_{\cf}$ satisfies that for
    every word $\mathbf{v}$ of $t$ digits all less than or equal to
    $t$, $\D{\cf-ary}_{\mathbf{v},|\mathbf{w}|}(\mathbf{w})$ is less
    than $\epsilon - (t-1)/|\mathbf{w}|$.

  \item[-] for each $b = 2,\ldots,t$ the new word $\mathbf{w}$ in base
    $b$ corresponding to the inclusion $\tau_b \subset \sigma_b$ has
    simple discrepancy
    $\D{b-ary}_{\left|\mathbf{w}\right|}(\mathbf{w})$ less than
    $\epsilon$.
  \end{enumerate}

\end{definition}

Notice that if $t'>t$ the definition of a refinement of a $t$-brick
gives no condition on $\tau_{t'}$.

\subsection{Lemmas}
The construction consists in choosing a sequence of nested intervals of
each type.  En each case the subinterval to be chosen is independent
of the subintervals chosen in previous steps.  We need to control the
discrepancy of the word representing the subinterval and the size of
the interval, which should be  larger than the measure of the bad
zones. By bad zones we mean  the \cf-ary and $b$-ary intervals corresponding to 
words with large discrepancy.

\subsubsection{On the length of continued fraction intervals}
We start by considering the length of the different continued fraction
intervals. For $x=[a_1,a_2,\ldots]$ we recursively define the
functions $p_{n}(x)$ and $q_{n}(x)$, called the convergents of $x$, as
follows.  We set $p_{-1}(x)=q_0(x)=1$ and $p_0(x)=q_{-1}(x)=0$ and
recursively for $n\geq 1$,
\begin{gather*}
p_n(x)= a_n p_{n-1}(x) + p_{n-2}(x),
\\
q_n(x)= a_n q_{n-1}(x) + q_{n-2}(x).
\end{gather*}
For irrational $x=[a_1, a_2, \ldots]$, $p_n(x)/q_n(x)$ is the $n$th approximant
to $x$ and converges to $x$ as $n$ tends to infinity. 
For rational $x=[a_1, \dotsc, a_n]$, we have that
$x=p_n(x)/q_n(x)$ and we  write $q(x)$ to denote
$q_n(x)$. Observe that for every $x$,
 $(p_n (x))_{n \geq 1}$ and $(q_n(x))_{n \geq 1}$
are increasing. Furthermore, the length of a \cf-ary interval is
\[
\left|I_{[a_1,\dotsc,a_n]}\right| = \frac{1}{q_n(q_n + q_{n-1})}.
\]

%In particular, for $a_1=1$ we have that $2|i|=|I|$, which is why we
%fixed $2|\sigma_{\cf}|=|\Sigma_{\cf}|$ in the construction (note that
%any other choice of $a_1$ would be useful).
%\\
%For instance 
%$[1,1,2] = \frac{3}{5} =\frac{21}{35}$ and  $[1, 1,3]=\frac{ 4}{7}= \frac{20}{35}$
%and $I_{[1,1,2]}=(\frac{4}{7}, \frac{3}{5})$
%and length interval $ I_{[1,1,2]} $ is $ \frac{1}{35}   $
%\\
%However, 
%$[1,2]=\frac{ 2}{3}=\frac{8}{12}$  and $[1,3]=\frac{3}{4}=\frac{ 9}{12}$, \\
%and the  length of interval  $I_{[ 1,2]}$ is $\frac{1}{12} $ .\\
%And $2\frac{1}{35} \not =\frac{1}{12} $.
%}
%\newpage

%Let $x=[a_1,\dotsc,a_r]$, $y=[a_{r+1},\dotsc,a_s]$ and
%$z=[a_1,\dotsc,a_s]$. Then we write $I_{x,y}$ to denote $I_z$.  

% \begin{lemma}[\protect{\cite[Lemma 3]{BY}}]   \label{lemma:cfrelation}   %\label{lemma:relative}
%  Let $x=[a_1,\dotsc,a_r]$, $y=[a_{r+1},\dotsc,a_s]$ and $z=[a_1,\dotsc,a_s]$. 
% Then,
% \begin{enumerate}
% \item \qquad  $ q(x) q(y) \ \leq q(z) \ \leq \ 2q(x) q(y)$ and
 
%  \item \qquad $|I_y| / 2 \ \leq \ {|I_{z}|}/{|I_x|}\ \leq \ 2|I_y|$.
% \end{enumerate}
% \end{lemma}

\begin{lemma}   \label{lemma:cfrelation}   %\label{lemma:relative}
  For $n\in\mathbb{N}$ and $a_i\in\mathbb{N}$ for $2\leq i\leq n$ we have
\[
| I_{[0;a_2, .., a_n]}|  / 4 \leq |I_{[0;1,a_2 .., a_n]} |\leq  | I_{[0;a_2, .., a_n]}| .
\]
\end{lemma}

\begin{proof}
  This is a special case of \cite[Lemma 3]{BY}.
\end{proof}

%By Lemma~\ref{lemma:relative}, we obtain the following.
%
%\begin{lemma} \label{lemma:reciprocals}
%For any rational number $x=[a_1, \ldots, a_n]$, the lengths 
%of the continued fraction interval  for $x$ and that for $1/x$ are related as follows:
%and 
%\[
%| I_{[0;a_2, .., a_n]}| \ | I_{[0;a_1]}| / 2 \leq |I_{[0;a_1, .., a_n]} |\leq 2| I_{[0;a_2, .., a_n]}|\ | I_{[0;a_1]}|.
%\]
%\end{lemma}

The distribution of $\log q_n$ obeys in the limit a Gaussian law.  It
was first proved by Ibragimov~\cite{Ibragimov1961}.  Then
Philipp~\cite[Satz 3]{Philipp1967} obtained an error term of
$O(n^{-1/5})$, which was later improved by
Mischyavichyus~\cite{Mischyavichyus1987} to $O(n^{-1/2} \log n )$.
Morita~\cite[Theorem 8.1]{Morita1994} obtained the optimal error term
of order $O(n^{-1/2})$; a different proof of the same bound was given
by Vall\'ee~\cite[Th\'eoreme 9]{Vallee1997}.  This allows for the
following lemma that ensures that there are many disjoint large
\cf-ary subintervals of relative order $n$ inside any given interval $I$.
This lemma is crucial for our construction.

%\begin{lemma}[\protect{Morita~\cite[Theorem 8.1]{Morita1994},
%Vall\'ee~\cite[Th\'eoreme 9]{Vallee1997}}]% \cite[Lemma4]{BY}}]
%\label{thm:tcl}
%The distribution of the random variable $\log q_n(x)$ is asymptotically Gaussian.
%There is $K_0$ and  $n_0$ such that for every~$n\geq n_0$,
%\[
%\left|
% \mbox{Pr} 
% \Big[ x \in (0,1) :  
%  -y \leq\frac{  \log q_n(x) -  n L  }{\sigma \sqrt{n}} \leq  y  \Big] -
% \frac{1}{\sqrt{2\pi}} \int_{-y}^y e^{-z^2/2} dz 
%\right|<
% \frac{K_0}{\sqrt n},
% \] 
%where 
%$L=\pi^2/(12 \log 2)$ is    L\'evy's constant 
%and 
%$\sigma$ is a positive absolute  constant. 
%\end{lemma}

We write $L$ for L\'evy's constant $\pi^2/(12 \log 2)$.

\begin{lemma}[\protect{\cite[Lemma 5]{BY}}]\label{lemma:cfgrande}
 There are positive constants $K, C$ and a positive integer $N_1$ 
such that for any \cf-ary interval $I$ 
and any  integer $n\geq N_1$,
 the Lebesgue measure of the union of the \cf-ary subintervals  $J$ of $I$ 
of relative order $n$ is such that 
\[
\frac{|I|}{4}e^{-2nL-2C}\leq |J| \leq  2|I|  e^{-2nL+2C}
\]
 is greater than 
\[
K|I|/\sqrt{n}.
\] 
\end{lemma}

\subsubsection{On the size of continued fraction intervals with large discrepancy}

The following result on large deviations is essentially Kifer, Peres and Weiss'
Corollary~3.2 in~\cite{KiferPeresWeiss2001} but conditioning the first~$r$ terms. 

\begin{lemma}[\protect{\cite[Lemma 6]{BY}}] \label{lemma:KPWc}
 Let $I_{[a_1,\dotsc,a_r]}$ be a \cf-ary interval, and let
 $\mathbf{v}=v_1\dotsc v_k$ be a word of length $k$ over the alphabet $\N$.
Then for every positive real $\delta$ and for every positive integer~$n$,
\[\left|\left\{x\in I_{[a_1,\ldots,a_r]}\colon
      \D{\cf-ary}_{\mathbf{v},n}(T^rx)>\delta\right\}\right|
  \leq 6 M e^{-\frac{\delta^2n}{2M}} |I_{[a_1\dotsc,a_r]}|, 
\]
 where 
\[
M= M(\delta, k) = \Big\lceil k - \log \left({\delta^2}/(2 \log 2) \right ) \Big\rceil.%, or any larger number.
\]
Recall that $T$ is the Gauss map.
\end{lemma}

\subsubsection{On the Discrepancy associated to continued fraction expansions}
If $\mathbf{v}$ and $\mathbf{u}$ are words, we write $\mathbf{vu}$
for their concatenation. Then the following lemma describes the change
of the discrepancy if we concatenate two words.

\begin{lemma}[\protect{\cite[Lemma 7]{BY}}]\label{lemma:Dcf}
Let   $\mathbf{w}=a_1\ldots a_n$, $\mathbf{u}= b_1\ldots b_s$ and
$\mathbf{v}=v_1\ldots v_k$ be finite words over the alphabet
$\N$. Furthermore, let $0<\epsilon<1$. Then,
\medskip

\begin{enumerate}
\item if $\D{\cf-ary}_{\mathbf{v},n}(\mathbf{w})<\epsilon$ and
  $\D{\cf-ary}_{\mathbf{v},s}(\mathbf{u})<\epsilon - (k-1)/s$ then
$\D{\cf-ary}_{\mathbf{v},n+s}(\mathbf{wu})<\epsilon ;$
\medskip

\item  if   $\D{\cf-ary}_{\mathbf{v},n}(\mathbf{w})<\epsilon$ and
$s/n<\epsilon$ then 
\medskip

\begin{enumerate}
\item for every $1\leq \ell\leq s$,
 $\D{\cf-ary}_{\mathbf{v},n+\ell}(\mathbf{wu})<2\epsilon$
 and
\medskip

\item 
$\D{\cf-ary}_{\mathbf{v},n+s}(\mathbf{uw})< 2\epsilon.$

\end{enumerate}
\end{enumerate}
\end{lemma}

\subsubsection{On the length of $b$-ary subintervals}

For any integer $b$ greater than or equal to $2$, 
 we say that an interval  $I$ is $b$-ary of order $k$, if it is of the form
\[I=\left( 
\frac{a}{b^k}, \frac{a+1}{b^k}\right)
\]
for some positive integer $k$ and an integer $a$ with $0\leq a< b^k$.
We write $order_b(I) = k$.  If $I$ is a union of two consecutive
$b$-ary intervals of the same order,
$I = \left(\frac{a}{b^k}, \frac{a+2}{b^k}\right)$, we also write
$order_b(I) = k$. We drop the index $b$ if the base is clear.
The following is a trivial fact about lengths of $b$-ary subintervals.

\begin{lemma}\label{lemma:ladrillo}
  Let $b \geq 2$ and $m \in \mathbb N$. Every interval $I$ whose
  Lebesgue measure is less than $b^{-m}$ is contained in a $b$-ary
  interval of order $m$ or in the union of two such intervals.
\end{lemma}

\subsubsection{On the number of $b$-ary words with large discrepancy}
In   the construction we use  the following  classical
bound for the number of blocks of a given length 
having larger discrepancy than a given value, 
see~\cite[Theorem 148]{HardyWright2008}
or~\cite[Lemma 2.5]{poly}.

\begin{lemma}[Bernstein inequality]\label{lemma:BHS}
  Let $I_{a_1,\ldots,a_r}$ be a $b$-ary interval
  For every positive integer* $n$ and for every real* $\delta$ such that $6/n \leq \delta \leq 1/b$ we have
\[
\left|\left\{x\in I_{a_1,\ldots,a_r}\colon
      \D{b-ary}_{n}(S_b^rx)>\delta\right\}\right|
  \leq 2b^{n+1}e^{-b\delta^2n/6} |I_{a_1\dotsc,a_r}|.
\]
Recall $S_bx= bx-\lfloor bx\rfloor$.
\end{lemma}

\subsubsection{On the Discrepancy associated to $b$-ary expansions} \label{sec:integer}
Since there are only finitely many digits in the $b$-ary expansion the
bounds for the discrepancy are easier in that case. 

\begin{lemma}[\protect{\cite[Lemma 3.1]{poly}}]\label{lemma:D}
  Let $\mathbf{u}$ and $\mathbf{v}$ be blocks in base~$b$ and  let $\epsilon > 0$.

\begin{enumerate}
 \item  If  $\D{b-ary}_{|\mathbf{u}|}(\mathbf{u})<\epsilon$ and
$\D{b-ary}_{|\mathbf{v}|}(\mathbf{v})<\epsilon$, then $\D{b-ary}_{|\mathbf{uv}|}(\mathbf{uv})< \epsilon$.

 \item If $\D{b-ary}_{|\mathbf{v}|}(\mathbf{v})<\epsilon$ and  $|\mathbf{u}|/|\mathbf{v}|<\epsilon$, then 
 \begin{enumerate}
 \item for every $\ell$ less than or equal to $|\mathbf{u}|$, 
 $\D{b-ary}_{|\mathbf{v}|+\ell}( \mathbf{v u} )< 2\epsilon.$
 \item $\D{b-ary}_{|\mathbf{v}|+|\mathbf{u}|}(\mathbf{uv} )< 2\epsilon$. 
\end{enumerate}
\end{enumerate}
\end{lemma}

\section{Proof of Theorem \ref{thm}}

We split the proof into three parts. First we construct $x$ and
$y:=1/x-\lfloor 1/x\rfloor$. Secondly we prove that $x$ and $1/x$ are
both continued fraction normal and absolutely normal. Finally we show
that both numbers are computable.

\subsection{The construction}

Iteratively we define two sequences of refinements of $t$-bricks
$\vec{\sigma}_1, \vec{\sigma}_2, \vec{\sigma}_3,\ldots$ and
$\vec{\Sigma}_1, \vec{\Sigma}_2, \vec{\Sigma}_3,\ldots$ for
non-decreasing values of~$t$. The intersection of all the intervals in
the first sequence defines the number $x$, whereas the intersection of
all the intervals in the second sequence defines the number $y$.

Before starting with the actual construction we provide a lemma
ensuring that the sequence of refinements of $t$-bricks exists.

\subsubsection{The refinement lemma}

\begin{lemma}\label{lemma:main}
Let $t$  be a positive integer
 greater than or equal to $1$, 
let $\epsilon$ be a positive real  less than~$1/t$
and let $t'$ be an integer equal to $t$ or to $t+1$.
Then, 
there is an integer  function $n_0=n_0(t, \epsilon)$ such that for every $n\geq  n_0$
and  
there are positive integers  $\ell_1, \ldots \ell_n$  such that 
 for any pair $t$-bricks 
$(\sigma_{\cf},\sigma_2,\dotsc,\sigma_t)$ and
$(\Sigma_{\cf},\Sigma_2,\dotsc,\Sigma_t)$ there are  refinements
$ (\tau_{\cf},\tau_2,\dotsc,\tau_{t'})$ and
$(\Tau_{\cf},\Tau_2,\dotsc,\Tau_{t'})$,
 both with discrepancy less than $\epsilon(s)$
satisfying the following:
\begin{align*}
\text{If }&\sigma_{\cf}=[1,a_2, \ldots, a_N] \text{  and } \Sigma_{\cf}=[a_2, \ldots, a_N] \text{   then }
\\
& \tau_{\cf} = [1,a_2, \ldots a_N, \ell_1, \ldots, \ell_n]
\text{ and } \Tau_{\cf} = [a_2, \ldots a_N, \ell_1, \ldots, \ell_n].
\end{align*}
\end{lemma}
\medskip

\begin{proof} 
  First, we assume that $t'=t$.  \smallskip

{\em \  Towards the length of $\tau_{\cf}$ and $\Tau_{\cf}$.} 
 For a \cf-interval $\alpha$ and a positive integer $n$ consider $\mathcal I_n( \alpha)$ the finite set of i\cf-ary subintervals $A$ 
of   $\alpha$ of relative order $n$ such that
\begin{equation}\label{A}
  \frac{1}{4} e^{-2nL-2C} \leq \frac{|A|}{|\alpha|}\leq 2\ e^{-2nL+2C}.
\end{equation}
Let  $K, C, N_1$ be the constants provided by
Lemma~\ref{lemma:cfgrande}.
Then,  if $n \geq N_1$,
\begin{equation*}
\frac{  \left|\bigcup_{A \in \mathcal I_n(\alpha)}A \right|}{|\alpha|} \geq \frac{K}{\sqrt{n}}.
\end{equation*}
For each $n$, consider the sets $\mathcal I_n(\sigma_{\cf})$ and
$\mathcal I_n(\Sigma_{\cf})$. Note that by our choice of
$\sigma_{\cf}$ and $\Sigma_{\cf}$ these sets have the same cardinality
and there is a one-to-one correspondence between the elements by
adding the digit $1$ to those in the set $\mathcal{I}_n(\Sigma_{\cf})$.
At the end of the proof we will determine a value $n_0$ for $n$ and we will choose 
 $\tau_{\cf}$ in   $\cc I_n(\sigma_{\cf})$ and $\Tau_{\cf}$ in  $\cc I_n(\Sigma_{\cf})$ such that
\begin{equation*}\label{AT}
\frac{1}{4} e^{-2{n_{0}}L-2C} |\sigma_{\cf}|    \leq |\tau_{\cf}| \leq 2 e^{-2{n_{0}}L+2C} |\sigma_{\cf}| 
\end{equation*}
\begin{equation*}
\frac{1}{4} e^{-2{n_{0}}L-2C} |\sigma_{\cf}|  \leq   |\Tau_{\cf}| \leq 2 e^{-2{n_{0}}L+2C} |\Sigma_{\cf}|.
\end{equation*}
And  by Lemma~\ref{lemma:cfrelation} we have
 %\[2|\tau_{\cf}|=|\Tau_{\cf}|.\]
\[ |\Tau_{\cf}|/4 \leq |\tau_{\cf}|\leq |\Tau_{\cf}|.
\]

{\em Towards the length of  $\tau_b$ and $\Tau_b$.}
For  each $b=2, \ldots, t$  we call 
\[
  m_b= order_b(\Tau_b) = order_b(\tau_b).
\] 
By the definition of a $t$-brick we have
\begin{equation}\label{eq0}
  |\Tau_{\cf}| \leq   b^{-m_b}.
\end{equation}
We choose $m_b$ as the largest integer such that 
\[
2e^{-2nL+2C} |\Sigma_{cf}| \leq b^{-m_b}.
\]
Thus
\[
b^{-m_b-1}< 2e^{-2nL+2C} |\Sigma_{cf}|.
\]
Using  the leftmost inequality in \eqref{A} we obtain,
 \begin{equation}
\label{B}
b^{-m_b-1} < 2e^{-2nL+2C} |\Sigma_{cf}|\  =8 e^{4C} \frac{1}{4} e^{-2nL-2C} |\Sigma_{cf}| \leq 8 e^{4C} |I|.
\end{equation}

For every $i\in \mathcal I_n(\sigma_{\cf})$ and for the corresponding
$I \in \mathcal I_n(\Sigma_{\cf})$  we have
\[
|I|/4\leq |i|\leq |I| 
\]
and from \eqref{B} we obtain
 \begin{equation}
\label{Bs}
b^{-m_b-1} < 4\cdot 8 e^{4C}|i|.
\end{equation}

Then, for each $i\in \mathcal I_n(\sigma_{\cf})$ and for the corresponding
$I \in \mathcal I_n(\Sigma_{\cf})$ we respectively determine
$\tau_b^i$ and $\Tau_b^I$ as the $b$-ary intervals of order $m_b$ or
the union of two consecutive $b$-ary intervals of order $m_b$  
 that
respectively contain $i$ and $I$ (Lemma~\ref{lemma:ladrillo}) with the
same choice for $\tau_b^i$ and $\Tau_b^I$. 
Thus, either   $|\tau_b^i|=|\Tau_b^I|=b^{-m_b}$ or 
 $|\tau_b^i|=|\Tau_b^I|=2 b^{-m_b}$.
Putting together  \eqref{eq0}, \eqref{B} and \eqref{Bs} we obtain
\begin{equation}\label{f}
\frac{1}{2\cdot 8 e^{4C}b}\ \leq \ \frac{|I|}{|\Tau_b^I|}\ \leq \ \frac{4|i|}{|\tau_b^i|} .
\end{equation}

We give bounds on the number of digits we add
in the $b$-ary expansion. For this we write, 
\[
n_b=order(\Tau_b)-order(\Sigma_b).
\]
Since 
\begin{equation*}
 |\Sigma_{cf}| \leq |\Sigma_b| \leq |\Sigma_{cf}| 2\cdot 8e^{4C} b
\end{equation*}
and  by Lemma~\ref{lemma:ladrillo},
$\Sigma_b$ consists of one
or two $b$-ary intervals,
\[
order(\Sigma_b)=
-\log_b(|\Sigma_b|)  \ \ \ \text{ or } \ \ \ 
order(\Sigma_b)=-\log_b\left(|\Sigma_b|/2\right),
\]
we have
\[
 \log_b\left(|\Sigma_{cf}|/2\right) \leq -order(\Sigma_b) \leq \log_b(|\Sigma_{cf}| 8e^{4C} b).
\]
And  since \ \
\begin{equation*}
2e^{-2nL+2C} |\Sigma_{cf}| \leq b^{-m_b} \leq b \ 2e^{-2nL+2C} |\Sigma_{cf}|
\end{equation*}
we have
\[
 \log_b(2e^{-2nL+2C} |\Sigma_{cf}|) \leq -order(\Tau_b)=-m_b \leq \log_b(b \ 2e^{-2nL+2C}|\Sigma_{cf}|).
\]
We obtain, for the number of digits $n_b$ we add to the $b$-ary expansion, that
\[
 2nL \log_b e - \log_b(4 b e^{2C}) \leq order(\Tau_b) -
 order(\Sigma_b) = n_b\leq 2nL\log_be+\log_b(4 e^{2C}b).
\]
Thus,
\begin{equation}
\label{n_b}
 2n \frac{L}{\log b} - \frac{2C}{\log b} - 3 \leq n_b \leq 2n \frac{L}{\log b} + \frac{2C}{\log b} + 3.
\end{equation}

{\em \ Bad zones.}  
We must pick one interval $i$ in $\mathcal I_n(\sigma_{\cf})$
 and one interval $I$ in $\mathcal I_n(\Sigma_{\cf})$ in a zone of low discrepancy. 
This is possible because  the measure of the   zones of large discrepancy 
decrease at an exponential rate in~$n$ 
while the measure of $\mathcal I_n(\sigma_{\cf})$  and $\mathcal I_n(\Sigma_{\cf})$  
decreases only as  $K / \sqrt{n}$. 
For each $n$ let   
\[
B^0_{b, \sigma_b, m_b, \epsilon} \text{ and }\  B^0_{b, \Sigma_b, m_b, \epsilon}
\] 
be the set of reals in the $b$-ary subintervals of $\sigma_b$ and
$\Sigma_b$ of order $m_b$ with $b$-discrepancy greater
than~$\epsilon$, respectively. And let
\[
B_{b, \sigma_b, m_b, \epsilon} \ \text{ and }\ \ 
B_{b, \Sigma_b, m_b, \epsilon}
\]  
be, respectively,  the union of   $B^0_{b, \sigma_b, m_b, \epsilon}$  and 
$B^0_{b, \Sigma_b, m_b, \epsilon}$ with those numbers
lying in a $b$-ary interval of the same order that is a neighbour to one in
$B^0_{b, \sigma_b, m_b, \epsilon}$ and $B^0_{b, \Sigma_b, m_b, \epsilon}$.

Recall that $m_b$ is the order of $\tau_b$ and $\Tau_b$, which we
reach by adding $n_b$ digits to the intervals $\sigma_b$ and
$\Sigma_b$, respectively. To define $\tau_b$ we need to add $n_b$ many
digits avoiding $b$-discrepancy greater than $\epsilon$.  Thus, using
the conditions $6/n_b \leq \epsilon \leq 1/b$, Lemma~\ref{lemma:BHS}
provides the estimate
\[
 \frac{|B_{b,\sigma_b,m_b,\epsilon}|}{|\sigma_b|} =
 \frac{|B_{b,\Sigma_b,m_b,\epsilon}|}{|\Sigma_b|} \leq 6b e^{-b   \epsilon^2 n_b / 6}.
\]
Notice that the factor $6$  on the rightmost expression 
comes from considering  the $b$-ary intervals in $B$ which are 
those in $B^0$  together with  their neighbour  $b$-ary intervals to the left and to the right.

By \eqref{f},
\begin{align*}
|\sigma_b| &\leq 4 |\sigma_{cf}|\cdot  2\cdot8 e^{4C} b, 
\\
|\Sigma_b| &\leq |\Sigma_{cf}| \cdot 2\cdot8 e^{4C} b
\end{align*}
and from \eqref{n_b} we know
\[
n_b \geq 2n \frac{L}{\log b} - \frac{2C}{\log b} - 3.
\]
We obtain
 \begin{align*}
 \frac{|B_{b,\sigma_b,m_b,\epsilon}|}{|\sigma_{\cf}| } \leq \frac{4|B_{b,\Sigma_b,m_b,\epsilon}|}{|\Sigma_{\cf}| } 
%\leq\ 4 \cdot 16 e^{4c} d \frac{|B_{d,\Sigma_d,m_d,\epsilon}|}{|\Sigma_d|} \\
% \leq \ 4\cdot 16\cdot\ 6 \cdot  e^{4c} d^2 e^{-d \epsilon^2 n_d / 6} 
\leq A(b) e^{-b \epsilon^2 L n / (3 \log b)},
\end{align*}
where
\[
A(b)=\ 384e^{4c} b^2    e^{b\epsilon^2 \left(\frac{C}{3\log b}+\frac{1}{2}\right)}.
\]

Consider the bad zones with respect to the continued
fraction expansion. 
For each $n$,    let   
\[
\tilde B_{t, \Sigma_{\cf}, n, \epsilon} \text{ and } \tilde B_{t, \sigma_{\cf}, n, \epsilon}
\]
be the set of reals $x$ in the respective \cf-ary subintervals of
$\Sigma_{\cf}$ and $\sigma_{\cf}$ of relative order $n$ such that for
some word of length $t$ of digits less than or equal to $t$ the
\cf-discrepancy of $x$ is greater than $\epsilon-(t-1)/n$.  With the
condition $2(t-1)/\epsilon \leq n$, it suffices to consider
\cf-discrepancies greater than $\epsilon/2$.  Then
Lemma~\ref{lemma:KPWc} gives the estimate,
 \[
\frac{|\tilde B_{t, \Sigma_{\cf}, n, \epsilon}|}{|\Sigma_{\cf}|} =
\frac{|\tilde B_{t, \sigma_{\cf}, n, \epsilon}|}{|\sigma_{\cf}|} 
\leq t^t 6Me^{-\frac{(\epsilon/2)^2 n}{2M}},
\]
 where 
\[
M = \left\lceil t - \log\left( \frac{(\epsilon/2)^2}{2\log 2}\right) \right\rceil.
\]
\medskip

{\em \ Find $n_0$ large enough.}  We choose $n_0$ such that
 the measure of the union of the bad zones of
$\sigma_{\cf}$ and $\Sigma_{\cf}$ as well as the bad zones of
$\sigma_b$ and $\Sigma_b$ for $b=2,\ldots, t$ are small
enough so that we can find and interval in $\cc I_n(\sigma_{\cf})$ and interval 
in $\cc I_n(\Sigma_{\cf})$ outside the bad zones 
and defined by appending the same   $n_0$ \cf-digits to the 
\cf-expansion of $\Sigma_{\cf}$ and $\sigma_{\cf}$.

We find $n_0$ to be the least  integer $n$ such that   for each $b=2,\ldots,t$, 

\begin{align*}
A(b)  e^{-b \epsilon^2 L n / (3 \log b)}
& < \frac{1}{8} \frac{K}{t \sqrt{n}} \quad\text{and}\\
6M t^t e^{-\frac{(\epsilon/2)^2 n}{2M}} 
&< \frac{1}{8} \frac{K}{t\sqrt{n}},
\end{align*}

where  the factor $1/8$ ensures that 
\begin{enumerate}
\item[1.]  less than  $1/8$ of the measure of $\cc I_n(\sigma_{\cf})$ is covered with bad zones
  with respect to the continued fraction expansion of $\sigma_{\cf}$;
\item[2.] less than $1/8$ of the measure of $\cc I_n(\sigma_{\cf})$ is
  covered with the projection of the bad zones with respect to the
  corresponding continued fraction expansion of $\Sigma_{\cf}$;
\item[3.] less than  $1/8$ of the measure of $\cc I_n(\sigma_{\cf})$  is covered with bad zones
  with respect to the $b$-ary expansion inside $\sigma_{\cf}$ and inside $\Sigma_{\cf}$;

\item[4.] at least $5/8$ of the measure of $\cc I_n(\sigma_{\cf})$ is free of bad zones,
\item[5.]  the above 4 points also hold on interchanging
  $\sigma_{\cf}$ with $\Sigma_{\cf}$.
\end{enumerate}
In turn, this ensures the existence  of $n$ digits $\ell_1, \ldots, \ell_n$ such that 
\begin{itemize}
\item if $\sigma_{\cf}=[1, a_2, \ldots, a_N]$ and $\Sigma_{\cf}=[a_2, \ldots, a_N]$
then $\tau_{\cf}=[1, a_2, \ldots, a_N, \ell_1,\ldots,  \ell_n]$ and 
$\Tau_{\cf}=[a_2, \ldots, a_N, \ell_1,\ldots, \ell_n]$,
 \item $\tau_{\cf }\in \cc I_n(\sigma_{\cf})$,
 $\Tau_{\cf }\in \cc I_n(\Sigma_{\cf})$,
\item  $\tau_{cf}$ and $\Tau_{\cf}$  are not in bad zones.
\end{itemize}
So, we need to find  solutions to 
 \[
\sqrt{n} e^{-rn} \leq \gamma
 \]
for certain values of $r$  and $\gamma$.
Since for every positive $x$,  it holds that  $x < e^{x/2}$, we have
 \[
\sqrt{n} e^{-r n/2} \leq \frac{1}{r} r\ n\ e^{-rn/2} \ <
 \frac{1}{r}\ e^{r n/2 - r n/2} = \frac{1}{r}.
\]
Thus, we need $n$ such that 
\[
e^{-r n/2} \leq \gamma r
\] 
for each of the needed values $r$ and $\gamma$.
Hence, $n$ has to be as large as
\[
 -2/ r \ \log(\gamma r)
\]
for each of the needed values $r$ and $\gamma$.
Letting
\[
r^{(1)}=\epsilon^2/(8M)  \text{ and } \gamma^{(1)}=K/(6 M t^{t+1})
\]
and for $b=2, \dotsc, t$,
\[
r^{(b)}=b \epsilon^2 L/(3\log b) \text{ and } \gamma^{(b)}=K/(t\  A(b)),
\]
Taking 
\[
n = \max \left\{   -2/ r^{(b)} \log\left(\gamma^{(b)} r^{(b)}\right): 1\leq b\leq t\right\}
\cup \left\{\frac{6}{\epsilon},\frac{2(t-1)}{\epsilon},N_1 \right\}
\]
(recall that $N_1$ is the constant already fixed at the beginning of
this proof, provided by Lemma~\ref{lemma:cfgrande}) completes the
proof in case $t'=t$.  \medskip

The case $t'=t+1$ follows easily by taking first
$t$-bricks  $\vec \tau=(\tau_{\cf}, \tau_2, \ldots \tau_{t})$  and $\vec \Tau=(\Tau_{\cf}, \Tau_2, \ldots \Tau_{t})$ respectively refining 
$\vec \sigma$ and $\vec \Sigma$ with discrepancy less than $\epsilon$.
Since the refinement asks no discrepancy condition on $\tau_{t+1}$ nor on $\Tau_{t+1}$,
we only need to take $(t+1)$-ary intervals  $\tau_{t+1}$  and $\Tau_{t+1}$  
of order  $m_{t+1}$, or a union of two consecutive such intervals so that 
$|\tau_{t+1}|=|\Tau_{t+1}|$,  
$\tau_{cf} \subset \tau_{t+1}$,
$\Tau_{cf} \subset \Tau_{t+1}$ 
where  $m_{t+1}$  is  the maximum such that 
$|\Tau_{cf}| \leq (t+1)^{-m_{t+1}}$.
Applying  Lemma~\ref{lemma:ladrillo}  and using that   $|\tau_{cf}|\leq |\Tau_{cf}|$
we obtain
\[
|\tau_{cf}| \geq \frac{|\tau_{t+1}|}{2(t+1)} \text{ and }
|\Tau_{cf}| \geq \frac{|\Tau_{t+1}|}{2(t+1)}.
\]
This ensures that   $\vec \Tau=(\Tau_{\cf}, \Tau_2, \ldots, \Tau_{t+1})$
and    $\vec \tau=(\tau_{\cf}, \tau_2, \ldots, \tau_{t+1})$ are $(t+1)$-bricks.\qedhere
\end{proof}

\subsubsection{The iterative construction}
For simplicity, we fix the first digit of the continued fraction
expansion such that $x$ is in the interval $(1/2,1)$ and therefore
$1/x$ is in interval $(1,2)$, hence $\lfloor 1/x\rfloor=1$ and $y $ is
in interval $(0,1)$.

The construction works by steps, starting at step~$s=1$. We set the
largest integer base $t$, the discrepancy value $\epsilon$ and the
relative order $n$ of the new \cf-ary interval as functions of the
step~$s$. In particular, we define for every positive integer $s$,
\begin{align*}
t(s)&=\max(2, \lfloor \sqrt[5]{ \log s} \rfloor),
\\
\epsilon(s)&=\ 1/t(s).
\end{align*}
Clearly $t(s)$ is non-decreasing unbounded 
and $\epsilon(s)$ is  non-increasing and  goes to zero.
Consider the function  $n_0\big(\epsilon(s), t(s)\big)$ given by Lemma~\ref{lemma:main}  below
and notice that 
\[
n_0\big(\epsilon(s), t(s)\big)=O\big(t(s)^4 \log (t(s))\big).
\]
Let $n_{\start}$ be the minimum positive integer such that for every positive $s$,
\[
\lfloor \log s \rfloor + n_{\start} \geq n_0(\epsilon(s), t(s))
\]
and define
\[
n_0(s)= \lfloor \log s \rfloor + n_{\start}.
\]
The following is invariant in all steps $s$  of  the construction for $\vec \sigma_s=(\sigma_{cf}, \sigma_2, \ldots \sigma_{t(s)})$ and 
$\vec \Sigma_s=(\Sigma_{cf}, \Sigma_2, \ldots \Sigma_{t(s)})$:
\[
 |\Sigma_{\cf} |/4 \leq |\sigma_{\cf} |\leq  |\Sigma_{\cf} |,
\]
and for each $b=2, \ldots t(s)$,
\begin{align*}
|\sigma_b|&=|\Sigma_b|
\\
 \sigma_{\cf} &\subset \sigma_b,
\\
 \Sigma_{\cf} &\subset \Sigma_b,
\\
|\sigma_{\cf}| &\geq {|\sigma_b|}/({4\cdot 16 e^{4C}b}),
\\
|\Sigma_{\cf}| & \geq {|\Sigma_b|}/({16e^{4C}b}).
\end{align*}

\begin{description}
\item[\textnormal {\em Initial step, $s=1$}] 
 \begin{align*}
  \vec{\sigma}_1&=( \sigma_{\cf}, \sigma_2), \text{ for } 
  \sigma_2=\sigma_{\cf}=(1/2,1)=I_{[0;1]}  
\\ \vec{\Sigma}_1&=( \Sigma_{\cf}, \Sigma_2), \text{ for }
  \Sigma_2=\Sigma_{\cf}=(0,1)=I_{[0;]}.
\end{align*}

\item[\textnormal {\em Recursive step, $s>1$}] 
Assume that we already have two bricks
\[
\vec{\sigma}_{s-1}=(\sigma_{\cf }, \sigma_2,\dotsc,\sigma_{t(s-1)} )
  \quad\text{and}\quad
\vec{\Sigma}_{s-1}=(\Sigma_{\cf }, \Sigma_2,\dotsc,\Sigma_{t(s-1)} ).
\]
We choose
\begin{align*}
\vec{\sigma}_s &=(\tau_{\cf},\tau_2,\dotsc,\tau_{t(s)})
\\
\vec{\Sigma}_s &=(\Tau_{\cf},\Tau_2,\dotsc,\Tau_{t(s)})  
\end{align*}
 such that 
if $\sigma_{\cf}=[a_1, \ldots, a_N]$ and 
$\Sigma_{\cf}=[a_2, \ldots, a_N]$  then 
\begin{align*}
\tau_{\cf} &= [a_2, \ldots a_N, a_{N+1}, \ldots a_{N+n_0(s)}],
\\
\Tau_{\cf} &= [a_1, \ldots a_N, a_{N+1}, \ldots a_{N+n_0(s)}]
\end{align*}
are the  leftmost \cf-subintervals of $\sigma_{\cf}$ and $\Sigma_{\cf}$ of relative order $n_0(s)$
ensuring that  $\vec{\sigma}_{s}$ refines  $\vec{\sigma}_{s-1}$ and 
$\vec{\Sigma}_{s}$ refines $\vec{\Sigma}_{s-1}$,
both with discrepancy less than $\epsilon(s)$.
\end{description}

\subsection{Correctness of the construction}
The existence of the sequences
$\vec{\sigma}_1, \vec{\sigma}_2, \ldots$ and
$\vec{\Sigma}_1, \vec{\Sigma}_2, \ldots$ is guaranteed by
Lemma~\ref{lemma:main}.  Let $x$ and $y$ be respectively defined by
the intersection of all the intervals in the respective sequences.
%The construction ensures that, removing the first digit in the
%continued fraction expansion of $x$, the  continued fractions of
%$x$ and $y$ are identical.

\subsubsection{The numbers $x$ and $1/x$ are continued fraction normal}
The construction ensures that, removing the first digit in the
continued fraction expansion of $x$, the  continued fractions of
$x$ and $y$ are identical.
Since $y =1/x -\lfloor 1/x\rfloor =1/x -1$, 
%the continued fraction
%expansion of $1/x$ is identical to that of $y$ except for an extra $1$
%in the integer part. Therefore to
to show that $x$ and $1/x$ are continued
fraction normal it suffices to show that $x$ and $y$ are continued
fraction normal.

% Now we prove that that $x$ is  continued fraction normal, and so is $y$.
Let $v$ be a word of  $m$ integers $v_1,\dotsc,v_m$ 
and let $\tilde \epsilon > 0$.
Choose $s_0$ so that 
$m\leq t(s_0)$, $\max\{v_1, \dotsc,v_m\}\leq t(s_0)$ and  $\epsilon(s_0) \leq \tilde \epsilon / 4$.
At each step $s$ after $s_0$, 
the continued fraction expansions of  $x$ and  $y$
are constructed  by appending a word $u_s$ 
such that $|u_s| = n_0(s)$ and
\[
\D{\cf-ary}_{v,|u_s|}(u_s) \ < \ 
\epsilon(s) - \frac{t(s-1)-1}{|u_s|} \ < \ \epsilon(s) - \frac{m-1}{|u_s|}.
\]
  By Lemma~\ref{lemma:Dcf} (Item 1) applied several times, we obtain for every $s \geq s_0$:
\[
\D{\cf-ary}_{v,|u_{s_0} \ldots u_s|}(u_{s_0}u_{s_0+1} \ldots u_s) < \epsilon(s_0).
\]
  Next, by Lemma~\ref{lemma:Dcf} (Item 2b) there is $s_1$ sufficiently large such
  that for every $s \geq s_1$,
 \[
\D{\cf-ary}_{v,|u_1 \ldots u_s|}(u_1 \ldots u_s) < 2\epsilon(s_0).
\]  
  Since $n_0(s)$ grows logarithmically, the inequality
\[
n_0(s)\leq 2\epsilon(s_0)\sum_{j=1}^{s-1} n_0(j)
\]
  holds from certain point on.
  Hence, by Lemma~\ref{lemma:Dcf} (Item 2a), we have for every $s$ sufficiently large and 
for every $\ell$ such that  $|u_1 \ldots u_{s-1}| < \ell \leq |u_1\ldots u_s|$,
  \[
   \D{\cf-ary}_{v,\ell}(u_1 \ldots u_s) < 4 \epsilon(s_0) < \tilde \epsilon.
  \]
  It follows that  $x$ and $y$ are   continued fraction normal.

\subsubsection{The numbers $x$ and $1/x$ are absolutely normal}

Absolute normality follows by showing simple normality to all integer
bases greater than or equal to~$2$.  We prove that $x$ simply normal
to all integer bases $b\geq 2$, the case of $y$ is alike. Since
$1/x=\lfloor 1/x\rfloor +y =1+ y$, we conclude that $1/x$ is also
simply normal to all integer bases $b\geq 2$.

Fix a base $b\geq 2$ and let $\tilde \epsilon > 0$. We choose $s_0$ such
that $t(s_0) \geq b$ and $\epsilon(s_0)\leq \tilde \epsilon / 4$.  At
each step $s$ after $s_0$ the expansion of $x$ in base $b$ was
constructed by appending blocks $u_s$ such that
$\D{b-ary}_{|u_s|}(u_s)< \epsilon(s_0)$.  Thus, by Lemma~\ref{lemma:D}
(Item 1) for any $s > s_0$,
\[
\D{b-ary}_{|u_{s_0} \ldots u_s|}(u_{s_0} \ldots u_s) < \epsilon(s_0). 
\]
  Applying Lemma~\ref{lemma:D} (Item 2a), we obtain $s_1$ such that 
  for any $s > s_1$
\[
\D{b-ary}_{|u_1 \ldots u_s|}(u_1 \ldots u_s) < 2\epsilon(s_0).
\]
Call $n_b(j)$ the relative order of the $b$-interval of $\vec \sigma_{j}$
with respect to the $b$-interval of $\vec \sigma_{j-1}$. 
The inequalities 
\[ 
2n_0(j) \frac{L}{\log b} - \frac{2C}{\log b} - 3 
\leq n_b(j) \leq 2n_0(j) \frac{L}{\log b} + \frac{2C}{\log b} + 3 
\]
provided by \eqref{n_b} in the proof of Lemma~\ref{lemma:main}, tell us that
$n_b(j)$ grows logarithmically.
Then, for~$s$ sufficiently large we have
 \[
 n_b(s) \leq 2\epsilon(s_0) \sum_{j=1}^{s-1}  n_b(j).
 \]
 By Lemma~\ref{lemma:D} (Item 2b) we conclude that for~$s$ sufficiently large
 and $|u_1 \ldots  u_{s-1}| \leq \ell \leq |u_1 \ldots  u_s|$, 
\[
\D{b-ary}_{\ell}(u_1 \ldots  u_s)
<\ 4 \epsilon(s_0) < \tilde \epsilon.
\]
Thus $x$ is simply normal to base $b$, for every $b\geq 2$.

\subsection{The numbers $x$ and $1/x$ are computable}
A real number is computable if, for some integer $b\geq 2$, there is
an algorithm that produces, one after the other, the digits in its
$b$-ary expansion.  In addition to L\'evy's constant
$L=\pi^2/(12 \log 2)$, our construction of $x$ and $y$ depends on
three constants, $K, C$ and $N_1$ indicated in Lemma
\ref{lemma:cfgrande}.  Since these three constants can be taken to be
integer values (and they do not need to be minimal), there is an
algorithm that, for any given integer $b\geq 2$, produces the $b$-ary
expansion of $x$ and $1/x$.  Therefore, $x$ and $1/x$ are computable.
This completes the proof of Theorem~\ref{thm}.  \bigskip \bigskip

\noindent
{\bf Acknowledgements.}
The first author is supported by grant STIC-Amsud 20-STIC-06 and
PICT-2018-02315. The second author is supported by project
ANR-18-CE40-0018 funded by the French National Research Agency.  The
problem was brought to the authors' attention at the workshop on
``Discrepancy Theory and Applications -- Part 2'', CIRM 2641, February
2021, organized by Manfred Madritsch, Joel Rivat and Robert Tichy.

\bibliography{x1x}

% \bib, bibdiv, biblist are defined by the amsrefs package.
\begin{bibdiv}
\begin{biblist}

\bib{poly}{article}{
      author={Becher, V.},
      author={Heiber, P.A.},
      author={Slaman, T.},
       title={A polynomial-time algorithm for computing absolutely normal
  numbers},
        date={2013},
     journal={Information and Computation},
      volume={232},
       pages={1\ndash 9},
}

\bib{BCchapter}{incollection}{
      author={Becher, Ver\'onica},
      author={Carton, Olivier},
       title={Normal numbers and computer science},
        date={2018},
   booktitle={Sequences, groups, and number theory},
      editor={Berth\'e, Val\'erie},
      editor={editors, Michel~Rig\'o},
   publisher={Trends in Mathematics Series Birkhauser/Springer},
       pages={233\ndash 269},
}

\bib{BY}{article}{
      author={Becher, Ver\'onica},
      author={Yuhjtman, Sergio},
       title={On absolutely normal and continued fraction normal numbers},
        date={2019},
     journal={International Mathematics Research Notices},
      volume={19},
}

\bib{Birkhoff1931}{article}{
      author={Birkhoff, G.D.},
       title={Proof of the {E}rgodic {T}heorem},
        date={1931},
     journal={Proc. Nat. Acad. Sci.},
      volume={17},
       pages={656\ndash 660},
}

\bib{Borel1909}{article}{
      author={Borel, \'{E}.},
       title={Les probabilit\'{e}s d'enombrables et leurs applications
  arithm\'{e}tiques},
        date={1909},
     journal={Supplemento di Rendiconti del Circolo Matematico di Palermo},
      volume={27},
       pages={247\ndash 271},
}

\bib{Bugeaud2012}{book}{
      author={Bugeaud, Y.},
       title={Distribution modulo one and {D}iophantine approximation},
      series={Cambridge Tracts in Mathematics},
   publisher={Cambridge University Press},
     address={Cambridge, UK},
        date={2012},
      number={193},
}

\bib{DK}{book}{
      author={Dajani, K.},
      author={Kraaikamp, C.},
       title={Ergodic theory of numbers},
      series={Carus Mathematical Monographs},
   publisher={Mathematical Association of America},
     address={Washington, DC},
        date={2002},
      volume={29},
        ISBN={0-88385-034-6},
      review={\MR{1917322 (2003f:37014)}},
}

\bib{drmotatichy1997}{book}{
      author={Drmota, M.},
      author={Tichy, R.},
       title={Sequences, discrepancies and applications},
      series={Lecture Notes in Mathematics, Vol. 1651},
   publisher={Springer-Verlag},
        date={1997},
}

\bib{HardyWright2008}{book}{
      author={Hardy, G.~H.},
      author={Wright, E.~M.},
       title={An introduction to the theory of numbers},
     edition={Sixth},
   publisher={Oxford University Press},
     address={Oxford},
        date={2008},
}

\bib{Ibragimov1961}{article}{
      author={Ibragimov, I.},
       title={A theorem from the metric theory of continued fractions},
        date={1961},
     journal={Vestnik Leningrad. Univ.},
      volume={16},
      number={1},
       pages={13\ndash 24},
}

\bib{KiferPeresWeiss2001}{article}{
      author={Kifer, Y.},
      author={Peres, Y.},
      author={Weiss, B},
       title={A dimension gap for continued fractions with independent digits},
        date={2001},
     journal={Israel Journal of Mathematics},
      volume={124},
      number={1},
       pages={61\ndash 76},
}

\bib{Kuipers2006}{book}{
      author={Kuipers, L.},
      author={Niederreiter, H.},
       title={Uniform distribution of sequences},
   publisher={Dover Publications, Inc.},
     address={New York},
        date={2006},
}

\bib{lindmarcus1995}{book}{
      author={Lind, Douglas},
      author={Marcus, Brian},
       title={An introduction to symbolic dynamics and coding},
   publisher={Cambridge University Press},
     address={Cambridge},
        date={1995},
        ISBN={0-521-55124-2; 0-521-55900-6},
         url={http://dx.doi.org/10.1017/CBO9780511626302},
      review={\MR{1369092 (97a:58050)}},
}

\bib{madritsch2018}{incollection}{
      author={Madritsch, Manfred},
       title={Normal numbers and symbolic dynamics},
        date={2018},
   booktitle={Sequences, groups, and number theory},
      editor={Berth\'e, Val\'erie},
      editor={editors, Michel~Rig\'o},
   publisher={Trends in Mathematics Series Birkhauser/Springer},
       pages={271\ndash 329},
}

\bib{MST}{article}{
      author={Madritsch, Manfred~G.},
      author={Scheerer, Adrian-Maria},
      author={Tichy, Robert~F.},
       title={Computable absolutely {P}isot normal numbers},
        date={2018},
        ISSN={0065-1036},
     journal={Acta Arithmetica},
      volume={184},
      number={1},
       pages={7\ndash 29},
         url={https://doi.org/10.4064/aa8661-8-2017},
      review={\MR{3826637}},
}

\bib{Mischyavichyus1987}{article}{
      author={Mischyavichyus, G.A.},
       title={Estimate of the remainder in the limit theorem for the
  denominators of continued fractions},
        date={1987},
     journal={Litovskii Matematiceskii Sbomik},
      volume={21},
      number={3},
       pages={63\ndash 74},
}

\bib{Morita1994}{article}{
      author={Morita, T.},
       title={Local limit theorem and distribution of periodic orbits of
  {L}asota-{Y}orke transformations with infinite {M}arkov partitions},
        date={1994},
     journal={J. Math. Sot. Japan},
      volume={46},
      number={2},
       pages={309\ndash 343},
        note={Corrections in Vol. 47 (I) (1997) 191--192},
}

\bib{Philipp1967}{article}{
      author={Philipp, W.},
       title={Ein zentraler grenzwertsatz mit anwendungen auf die
  zahlentheorie},
        date={1967},
     journal={Wahrscheinlichkeitstheorie verw. Geb.},
      volume={8},
       pages={185\ndash 203},
}

\bib{Rivoal}{article}{
      author={Rivoal, Tanguy},
       title={On the bits counting function of real numbers},
        date={2008},
     journal={Journal Australian Mathematical Society},
      volume={85},
       pages={95\ndash 111},
}

\bib{Vallee1997}{article}{
      author={Vall\'ee, B.},
       title={Operateurs de {R}uelle-{M}ayer generalises et analyse des
  algorithmes d'{E}uclide et de {G}auss},
        date={1997},
     journal={Acta Arithmetica},
      volume={LXXXI.2},
       pages={101\ndash 144},
}

\end{biblist}
\end{bibdiv}

\end{document}